\documentclass[11pt,twoside]{amsart}
\usepackage{amsxtra}
\usepackage{color}
\usepackage{amsopn}
\usepackage{amsmath,amsthm,amssymb,arydshln}
\usepackage{mathrsfs,mathtools}
\usepackage{stmaryrd}
\usepackage{enumerate}
\usepackage{xcolor}
\usepackage{pifont}
\usepackage{adjustbox}
\usepackage{tikz}
\usepackage{color}
\definecolor{cobalt}{RGB}{61,89,171}
\usepackage[colorlinks,citecolor=cobalt,linkcolor=cobalt,urlcolor=cobalt,pdfpagemode=UseNone,backref = page
]{hyperref}

\newtheorem{theorem}{Theorem}[section]

\newtheorem{proposition}[theorem]{Proposition}

\newtheorem*{theorem*}{Theorem}
\theoremstyle{definition}
\newtheorem{definition}[theorem]{Definition}

\newtheorem{example}[theorem]{Example}
\newtheorem{remark}[theorem]{Remark}
\newtheorem{question}{Question}

\newcommand{\R}{{\mathbb R}}

\newcommand{\C}{{\mathbb C}}

\newcommand{\beq}{\begin{equation}}
\newcommand{\eeq}{\end{equation}}

\newcommand{\f}{\varphi}

\newcommand{\psip}{\psi}

\newcommand{\SU}{{\mathrm{SU}}}

\newcommand{\G}{{\mathrm G}}

\newcommand{\W}{\wedge}

\newcommand{\frg}{\mathfrak{g}}

\newcommand{\frh}{\mathfrak{h}}

\newcommand{\frn}{\mathfrak{n}}

\newcommand{\st}{\ |\ }

\newcommand{\sst}{\scriptscriptstyle}

\numberwithin{equation}{section}

\title[Special types of locally conformal closed G$_2$-structures]{Special types of locally conformal closed G$_{\mathbf2}$-structures}

\author{Giovanni Bazzoni}
\address{Departamento de \'Algebra, Geometr\'ia y Topolog\'ia \\ Universidad Complutense de Madrid\\ Plaza de Ciencias 3\\ 28040 Madrid\\ Spain}
\email{gbazzoni@ucm.es}

\author{Alberto Raffero}
\address{Dipartimento di Matematica ``G.~Peano'' \\ Universit\`a degli Studi di Torino\\ Via Carlo Alberto 10\\ 10123 Torino\\ Italy}
\email{alberto.raffero@unito.it}

\begin{document}
%%%%%%%%%%%%%%%%%%%%%%%%%%%%%%%%%%%%%%%%%%%%%%%%%%%%%%%%%%%%%%%%%%%%%%%%%%%%%%%%%%%%%%%%%
%%%%%%%%%%%%%%%%%%%%%%%%%%%%%%%%%%%%%%%%%%%%%%%%%%%%%%%%%%%%%%%%%%%%%%%%%%%%%%%%%%%%%%%%%
%																ABSTRACT 
%%%%%%%%%%%%%%%%%%%%%%%%%%%%%%%%%%%%%%%%%%%%%%%%%%%%%%%%%%%%%%%%%%%%%%%%%%%%%%%%%%%%%%%%%
%%%%%%%%%%%%%%%%%%%%%%%%%%%%%%%%%%%%%%%%%%%%%%%%%%%%%%%%%%%%%%%%%%%%%%%%%%%%%%%%%%%%%%%%%
\begin{abstract}
Motivated by analogous results in locally conformal symplectic geometry, we study different classes of G$_2$-structures defined by a locally conformal closed 3-form. 
In particular, we give a complete characterization of invariant exact locally conformal closed G$_2$-structures on simply connected Lie groups, 
and we present examples of compact manifolds with different types of locally conformal closed G$_2$-structures. 
\end{abstract}

\maketitle

%%%%%%%%%%%%%%%%%%%%%%%%%%%%%%%%%%%%%%%%%%%%%%%%%%%%%%%%%%%%%%%%%%%%%%%%%%%%%%%%%%%%%%%%%
%%%%%%%%%%%%%%%%%%%%%%%%%%%%%%%%%%%%%%%%%%%%%%%%%%%%%%%%%%%%%%%%%%%%%%%%%%%%%%%%%%%%%%%%%
%																INTRODUCTION 
%%%%%%%%%%%%%%%%%%%%%%%%%%%%%%%%%%%%%%%%%%%%%%%%%%%%%%%%%%%%%%%%%%%%%%%%%%%%%%%%%%%%%%%%%
%%%%%%%%%%%%%%%%%%%%%%%%%%%%%%%%%%%%%%%%%%%%%%%%%%%%%%%%%%%%%%%%%%%%%%%%%%%%%%%%%%%%%%%%%
\section{Introduction}

Over the last years, the study of smooth manifolds endowed with geometric structures defined by a differential form which is locally conformal to a closed one has attracted a great deal of attention. 
Particular consideration has been devoted to \emph{locally conformal K\"ahler} (\emph{LCK}) structures and their non-metric analogous, \emph{locally conformal symplectic} (\emph{LCS}) structures, 
see \cite{Bazzoni,Dragomir-Ornea,Ornea-Verbitsky,Vai} and the references therein. 
In both cases, the condition of being locally conformal closed concerns a suitable non-degenerate 2-form $\omega$, and is encoded in the equation $d\omega=\theta\W\omega$,  
where $\theta$ is a closed 1-form called the {\em Lee form}.  
LCK structures belong to the pure class $\mathcal{W}_4$ of Gray-Hervella's celebrated sixteen classes of almost Hermitian manifolds, see \cite{Gray-Hervella}. 
They are, in particular, Hermitian structures and their understanding on compact complex surfaces is related to the global spherical shell conjecture of Nakamura. 
As pointed out in \cite{Vai}, LCS geometry is intimately related to Hamiltonian mechanics. 
Very recently, Eliashberg and Murphy used $h$-principle arguments to prove that every almost complex manifold $M$ with a non-zero $[\theta]\in H^1_{dR}(M)$ admits a LCS structure 
whose Lee form is (a multiple of) $\theta$, see \cite{Eliashberg-Murphy}.

In odd dimensions, 7-manifolds admitting G$_2$-structures provide a natural setting where the locally conformal closed condition is meaningful. 
Recall that $\G_2$ is one of the exceptional Riemannian holonomy groups resulting from Berger's classification \cite{Berger}, and that a {\em $\G_2$-structure} on a 7-manifold $M$ is defined by a 3-form $\f$ 
with pointwise stabilizer isomorphic to G$_2$. Such a 3-form gives rise to a Riemannian metric $g_\f$ and to a volume form $dV_\f$ on $M,$ with corresponding Hodge operator $*_\f$.  
A G$_2$-structure $\f$ satisfying the conditions  
\begin{equation}\label{LCP}
d\f = \theta\W\f,\quad d*_\f\f = \frac43\,\theta\W*_\f\f, 
\end{equation}
for some closed 1-form $\theta$, is locally conformal to one which is both closed and coclosed. 
G$_2$-structures fulfilling \eqref{LCP}  correspond to the class $\mathcal{W}_4$ in Fern\'andez-Gray's classification \cite{FeGr}, and they are called \emph{locally conformal parallel} (\emph{LCP}), 
as being closed and coclosed for a G$_2$-form $\f$ is equivalent to being parallel with respect to the associated Levi Civita connection, see \cite{FeGr}.
It was proved by Ivanov, Parton and Piccinni in \cite[Thm.~A]{IPP} that a compact LCP $\G_2$-manifold is a mapping torus bundle over $\mathbb{S}^1$ with fibre a simply connected nearly K\"ahler manifold of 
dimension six and finite structure group. This shows that LCP $\G_2$-structures are far from abundant.

Relaxing the LCP condition by ruling out the second equation in \eqref{LCP} leads naturally to \emph{locally conformal closed}, a.k.a.~\emph{locally conformal calibrated} (\emph{LCC}), $\G_2$-structures. 
Also in this case, the unique closed 1-form $\theta$ for which $d\f=\theta\W\f$ is called \emph{Lee form}. 
LCC G$_2$-structures have been investigated in \cite{FFR,FeUg,FiRa1}; 
in particular, in \cite{FFR} the authors showed that a result similar to that of Ivanov, Parton and Piccinni holds for compact manifolds with a suitable LCC $\G_2$-structure. 
Roughly speaking, they are mapping tori bundle over $\mathbb{S}^1$ with fibre a 6-manifold endowed with a \emph{coupled} $\SU(3)$-structure, 
of which nearly K\"ahler structures constitute a special case. We refer the reader to Theorem \ref{StructResLCC} below for the relevant definitions and the precise statement.

In LCS geometry, one distinguishes between structures of the first kind and of the second kind (see \cite{BaMa,Vai}); 
the distinction depends on whether or not one can find an infinitesimal automorphism of the structure which is transversal to the foliation defined by the kernel of the Lee form. 
Another way to distinguish LCS structures is according to the vanishing of the class of $\omega$ in the Lichnerowicz cohomology defined by the Lee form. 
This leads to the notions of exact and non-exact LCS structures. 
A LCS structure of the first kind is always exact, but the converse is not true (see e.g.~\cite[Ex. 5.4]{BaMa}). The LCS structures constructed by Eliashberg and Murphy in \cite{Eliashberg-Murphy} are exact.

The purpose of this note is to bring ideas of LCS geometry into the study of LCC $\G_2$-structures. 
In Sections \ref{LCCsect} and \ref{LCCfirst}, after recalling the notion of conformal class of a LCC $\G_2$-structure, we consider exact structures, 
and we distinguish between structures of the first and of the second kind.
As it happens in the LCS case, the difference between first and second kind depends on the existence of a certain infinitesimal automorphism of the LCC G$_2$-structure $\f$ 
which is everywhere transversal to the kernel of the Lee form. 
As for exactness, every LCC G$_2$-structure $\f$ defines a class $[\f]_\theta$ in the Lichnerowicz cohomology $H^\bullet_\theta(M)$ associated with the Lee form $\theta$;  
$\f$ is said to be exact if  $[\f]_\theta=0\in H^3_\theta(M)$. 
As we shall see, LCC G$_2$-structures of the first kind are always exact, but the opposite needs not to be true (cf.~Example \ref{ex3}).

In the literature, there exist many examples of left-invariant LCP and LCC $\G_2$-structures on solvable Lie groups, see e.g.~\cite{Chiossi-Fino,FFR,FiRa1}. 
In the LCC case, the examples exhibited in \cite{FFR} admit a lattice, hence provide compact solvmanifolds endowed with an invariant LCC G$_2$-structure.
In Section \ref{LieAlgLCC}, we completely characterize left-invariant exact LCC G$_2$-structures on simply connected Lie groups: 
their Lie algebra is a rank-one extension of a six-dimensional Lie algebra with a coupled $\SU(3)$-structure by a suitable derivation (see Theorem \ref{ThmLCCex}). 
Moreover, using the classification of seven-dimensional nilpotent Lie algebras which carry a closed $\G_2$-structure by Conti and Fern\'andez \cite{CoFe}, 
we prove that no such nilpotent Lie algebra admits a LCC $\G_2$-structure (Proposition \ref{CnoLCC}). 
Finally, in Section \ref{exsect} we show that there exist solvable Lie groups admitting a left-invariant LCC G$_2$-structure which is not exact (see Example \ref{ExEinstein}). 
This is not true on nilpotent Lie groups, as every left-invariant LCC G$_2$-structure must be exact by a result of Dixmier \cite{Dixmier} on the Lichnerowicz cohomology.
We also show that, unlike in the LCS case, there exist exact LCC structures on unimodular Lie algebras which are not of the first kind (see Remark \ref{exact-not-1stkind}).

%%%%%%%%%%%%%%%%%%%%%%%%%%%%%%%%%%%%%%%%%%%%%%%%%%%%%%%%%%%%%%%%%%%%%%%%%%%%%%%%%%%%%%%%%
%%%%%%%%%%%%%%%%%%%%%%%%%%%%%%%%%%%%%%%%%%%%%%%%%%%%%%%%%%%%%%%%%%%%%%%%%%%%%%%%%%%%%%%%%
%																PRELIMINARIES 
%%%%%%%%%%%%%%%%%%%%%%%%%%%%%%%%%%%%%%%%%%%%%%%%%%%%%%%%%%%%%%%%%%%%%%%%%%%%%%%%%%%%%%%%%
%%%%%%%%%%%%%%%%%%%%%%%%%%%%%%%%%%%%%%%%%%%%%%%%%%%%%%%%%%%%%%%%%%%%%%%%%%%%%%%%%%%%%%%%%

\section{Preliminaries}
Let $M$ be a seven-dimensional manifold. A G$_2$-reduction of its frame bundle, i.e., a {\em $\G_2$-structure}, is characterized by the existence of a 3-form $\f\in\Omega^3(M)$ which can be pointwise 
written as 
\[
\left.\f\right|_p = e^{127} + e^{347}+ e^{567} + e^{135} - e^{146} - e^{236} - e^{245}, 
\]
with respect to a basis $(e^1,\ldots,e^7)$ of the cotangent space $T^*_pM.$  Here, the notation $e^{ijk}$ is a shorthand for $e^i \W e^j \W e^k$. 
A {\em $\G_2$-structure} $\f$ gives rise to a Riemannian metric $g_\f$ with volume form $dV_\f$ via the identity
\[
g_\f(X,Y)\, dV_\f = \frac16\, \iota_{\sst X}\f \W \iota_{\sst Y} \f\W\f,
\]
for all vector fields $X,Y\in\mathfrak{X}(M)$. We shall denote by $*_\f$ the corresponding Hodge operator. 

When a G$_2$-structure $\f$ on $M$ is given, the G$_2$-action on $k$-forms (cf.~\cite[Sect.~2]{Bry}) induces the following decompositions: 
\begin{eqnarray*}
\Omega^2(M) & = & \Omega^2_7(M) \oplus \Omega^2_{14}(M), \\
\Omega^3(M) & = & \mathcal{C}^\infty(M)\,\f \oplus \Omega^3_7(M) \oplus \Omega^3_{27}(M),
\end{eqnarray*}
where
\[
\Omega^2_7(M) \coloneqq \left\{\iota_{\sst X}\f \st X\in\mathfrak{X}(M) \right\},\quad \Omega^2_{14}(M) \coloneqq \left\{\kappa\in\Omega^2(M) \st \kappa \W *_\f\f=0\right\},\\
\]
\[
\Omega^3_7(M) \coloneqq \left\{*_\f(\f\W\alpha) \st \alpha\in\Omega^1(M)\right\},~\Omega^3_{27}(M) \coloneqq \left\{\gamma\in\Omega^3(M) \st \gamma\W\f=0,~\gamma\W*_\f\f=0\right\}.
\]
The decompositions of $\Omega^k(M)$, for $k=4,5$, are obtained from the previous ones via the Hodge operator. 

By the above splittings, on a 7-manifold $M$ endowed with a G$_2$-structure $\f$ there exist unique differential forms 
$\tau_0\in\mathcal{C}^\infty(M)$, $\tau_1\in\Omega^1(M)$, $\tau_2\in\Omega^2_{14}(M)$, $\tau_3\in\Omega^3_{27}(M)$, such that
\begin{equation}\label{ITF}
d\f = \tau_0*_\f\f + 3\,\tau_1\W\f + *_\f\tau_3,\quad d*_\f\f = 4\,\tau_1\W*_\f\f + \tau_2\W\f, 
\end{equation}
see \cite[Prop.~1]{Bry1}. 
Such forms are called {\em intrinsic torsion forms} of the G$_2$-structure $\f$, as they completely determine its intrinsic torsion.  
In particular, $\f$ is {\em torsion-free} if and only if all of these forms vanish identically, that is, if and only if $\f$ is both {\em closed} ($d\f=0$) and {\em coclosed} ($d*_\f\f=0$). 
When this happens, $g_\f$ is Ricci-flat and its holonomy group is isomorphic to a subgroup of $\G_2$. 
 
In this paper, we shall mainly deal with the G$_2$-structures defined by a 3-form which is locally conformal equivalent to a closed one. 
As we will see in Section \ref{LCCsect}, this condition corresponds to the vanishing of the intrinsic torsion forms $\tau_0$ and $\tau_3$. 
For the general classification of G$_2$-structures, we refer the reader to \cite{FeGr}. 

Since G$_2$ acts transitively on the 6-sphere with stabilizer $\SU(3)$, a $\G_2$-structure $\f$ on a 7-manifold $M$ induces an SU(3)-structure on every oriented hypersurface. 
Recall that an {\em $\SU(3)$-structure} on a 6-manifold $N$ is the data of an almost Hermitian structure $(g,J)$ with fundamental 2-form $\omega\coloneqq g(J\cdot,\cdot)$, and 
a unit $(3,0)$-form $\Psi = \psi + i \hat\psi$, where $\psi,\hat\psi\in\Omega^3(N)$.  
By \cite{Hit}, the whole SU(3)-structure $(g,J,\Psi)$ is completely determined by the 2-form $\omega$ and the 3-form $\psi = \Re(\Psi)$. 
In particular, at each point $p$ of $N$ there exists a basis $(e^1,\ldots,e^6)$ of the cotangent space $T^*_pN$ such that 
\[
\omega|_p = e^{12}+e^{34}+e^{56},\quad \psi|_p = e^{135} - e^{146} - e^{236} - e^{245}. 
\]

In a similar way as in the case of G$_2$-structures, the intrinsic torsion of an SU(3)-structure $(\omega,\psi)$ is encoded in the exterior derivatives $d\omega$, $d\psi$, $d\hat\psi$ (see \cite{ChSa}). 
According to \cite[Def.~4.1]{ChSa}, an SU(3)-structure is called {\em half-flat} if $d\omega\W\omega=0$ and $d\psi=0$. 
A half-flat SU(3)-structure is said to be {\em coupled} if $d\omega=c\,\psi$, for some $c\in\R\smallsetminus\{0\}$, while it is called {\em symplectic half-flat} if $c=0$, that is, 
if the fundamental 2-form $\omega$ is symplectic. We shall refer to $c$ as the {\em coupling constant}. 

If $h: N \hookrightarrow M$ is an oriented hypersurface of a 7-manifold $M$ endowed with a G$_2$-structure $\f$, and $V$ is a unit normal vector field on $N$, then 
the SU(3)-structure on $N$ induced by $\f$ is defined by the differential forms 
\[
\omega \coloneqq h^*(\iota_{\sst V}\f),\quad \psi \coloneqq h^*\f. 
\]
The reader may refer to \cite{MaCa} for more details on the relationship between G$_2$- and SU(3)-structures in this setting.

%%%%%%%%%%%%%%%%%%%%%%%%%%%%%%%%%%%%%%%%%%%%%%%%%%%%%%%%%%%%%%%%%%%%%%%%%%%%%%%%%%%%%%%%%
%%%%%%%%%%%%%%%%%%%%%%%%%%%%%%%%%%%%%%%%%%%%%%%%%%%%%%%%%%%%%%%%%%%%%%%%%%%%%%%%%%%%%%%%%
%																LCC DEFINITION 
%%%%%%%%%%%%%%%%%%%%%%%%%%%%%%%%%%%%%%%%%%%%%%%%%%%%%%%%%%%%%%%%%%%%%%%%%%%%%%%%%%%%%%%%%
%%%%%%%%%%%%%%%%%%%%%%%%%%%%%%%%%%%%%%%%%%%%%%%%%%%%%%%%%%%%%%%%%%%%%%%%%%%%%%%%%%%%%%%%%

\section{Locally conformal closed G$_2$-structures}\label{LCCsect}
A $\G_2$-structure $\f$ on a 7-manifold $M$ is said to be {\it locally conformal closed} or {\em locally conformal calibrated} ({\em LCC} for short) if  
\begin{equation}\label{lccdef}
d\f=\theta\W\f,
\end{equation} 
for some $\theta\in\Omega^1(M)$. Notice that such a 1-form is unique and closed, as the map 
\[
\cdot\W\f:\Omega^k(M)\rightarrow\Omega^{k+3}(M),\quad \alpha\mapsto \alpha\W\f,
\]
is injective for $k=1,2$. Moreover, it can be written in terms of $\f$ as follows 
\[
\theta= -\frac14*_\f\left(*_\f d\f\W\f \right),
\]
see \cite[Lemma 2.1]{FeUg}. 

\begin{definition} 
The unique closed 1-form $\theta$ fulfilling \eqref{lccdef} is called the {\it Lee form} of the LCC G$_2$-structure $\f$. 
\end{definition}

Henceforth, we denote a LCC G$_2$-structure $\f$ with Lee form $\theta$ by $(\f,\theta)$. As the name suggests, a LCC G$_2$-structure $(\f,\theta)$ is locally conformal equivalent to a closed one. 
Indeed, since $d\theta=0$, each point of $M$ admits an open neighborhood $\mathcal{U}\subseteq M$ where $\theta=df$, for some $f\in\mathcal{C}^\infty(\mathcal{U})$, and 
the 3-form $e^{-f}\f$ defines a closed G$_2$-structure on $\mathcal{U}$ with associated metric $e^{-\frac23 f}g_\f$ and orientation $e^{-\frac73 f}dV_\f$. 
Moreover, a LCC G$_2$-structure is globally conformal equivalent to a closed one when $\theta$ is exact, and it is closed if and only if $\theta$ vanishes identically. 

Given a LCC G$_2$-structure $(\f,\theta)$, we may consider its \emph{conformal class}
\[
\left\{e^{-f}\f \mid f\in \mathcal{C}^{\infty}(M)\right\}.
\]
It is easily seen that $\left(e^{-f}\f,\theta-df\right)$ is also LCC, hence the de Rham class $[\theta]\in H^1_{dR}(M)$ is an invariant of the conformal class.

\begin{remark}\ 
\begin{enumerate}[1)]
\item  The only non-identically vanishing intrinsic torsion forms of a LCC G$_2$-structure $(\f,\theta)$ are $\tau_1=\frac{1}{3}\theta$ and $\tau_2\in\Omega^2_{14}(M)$ (cf.~\eqref{ITF}). 
In particular
\[
d*_\f\f = \frac43\,\theta\W*_\f\f + \tau_2\W\f. 
\]
When $\tau_2$ vanishes identically, the G$_2$-structure is called {\em locally conformal parallel} (see \cite{Chiossi-Fino,IPP,Ver} for related results).

\item LCC G$_2$-structures belong to the class $\mathcal{W}_2\oplus\mathcal{W}_4$ in Fern\'andez-Gray classification \cite{FeGr}. 
The subclasses $\mathcal{W}_2$ and $\mathcal{W}_4$ correspond to closed and locally conformal parallel G$_2$-structures, respectively. 
\end{enumerate}
\end{remark}

Simple examples of manifolds admitting a LCC G$_2$-structure can be obtained as follows. 
Start with a 6-manifold $N$ endowed with a coupled SU(3)-structure $(\omega,\psi)$ such that $d\omega=c\psi$ (various examples can be found, for instance, in \cite{FiRa1,FiRa2,Raf}).  
Then, the product manifold $N\times\R$ admits a LCC G$_2$-structure given by the 3-form $\f=\omega\W dt+\psi$, 
where $dt$ denotes the global 1-form on $\R$. The Lee form of $\f$ is $\theta=-c\,dt$. 

More generally, if $(\omega,\psi)$ is coupled and $\nu\in\mathrm{Dif{}f}(N)$ is a diffeomorphism such that $\nu^*\omega=\omega$, then the quotient of $N\times\R$ by the infinite cyclic group 
of diffeomorphisms generated by $(p,t)\mapsto(\nu(p),t+1)$ is a smooth seven-dimensional manifold $N_\nu$ endowed with a LCC G$_2$-structure $\f$ (see \cite[Prop.~3.1]{FFR}). 
$N_\nu$ is called the {\em mapping torus} of $\nu$, and the natural projection $N_\nu\to\mathbb{S}^1$, $[(p,t)]\mapsto [t]$, is a smooth fibre bundle with fibre $N$. Notice that $N_{\mathrm{Id}} = N\times\mathbb{S}^1$.

In \cite{FeUg}, Fern\'andez and Ugarte proved that the LCC condition \eqref{lccdef} can be characterized in terms of a suitable differential subcomplex of the de Rham complex. In detail: 
\begin{proposition}[\cite{FeUg}]
A $\G_2$-structure $\f$ on a 7-manifold $M$ is LCC if and only if the exterior derivative of every 3-form in $\mathcal{B}^3(M)\coloneqq\mathcal{C}^\infty(M)\f\oplus\Omega^3_{27}(M)$ belongs to 
$\mathcal{B}^4(M)\coloneqq\Omega^4_7(M)\oplus\Omega^4_{27}(M)$. Consequently, $\f$ is LCC if and only if there exists the complex 
\[
0\rightarrow \mathcal{B}^3(M) \stackrel{\sst \hat{d}}{\rightarrow} \mathcal{B}^4(M) \stackrel{\sst \hat{d}}{\rightarrow}  \Omega^5(M) 
 \stackrel{\sst d}{\rightarrow}  \Omega^6(M) \stackrel{\sst d}{\rightarrow} \Omega^7(M)\rightarrow 0, 
\]
where $\hat{d}$ denotes the restriction of the differential $d$ to $\mathcal{B}^k(M)$, for $k=3,4$. 
\end{proposition}

As the Lee form $\theta$ of a LCC G$_2$-structure $\f$ is closed, it is also possible to introduce the {\em Lichnerowicz} (or {\em Morse-Novikov}) {\em cohomology} 
of $M$ relative to $\theta$. 
This is defined as the cohomology $H^\bullet_\theta(M)$ corresponding to the complex $(\Omega^\bullet(M),d_\theta)$, where   
\[
d_\theta : \Omega^k(M)\rightarrow\Omega^{k+1}(M),\quad d_\theta\alpha \coloneqq d\alpha-\theta\W\alpha. 
\]
It is clear that the condition \eqref{lccdef} is equivalent to $d_\theta\f=0$. Thus, $\f$ defines a cohomology class $[\f]_\theta\in H^3_\theta(M)$. 
If $[\f]_\theta=0$, namely if $\f=d_\theta\sigma$ for some $\sigma\in\Omega^2(M)$, then the LCC G$_2$-structure $\f$ is said to be {\em $d_\theta$-exact} or {\em exact}. 
Notice that being exact is a property of the conformal class of $\f$. 

More generally, if a G$_2$-structure $\f$ is $d_\theta$-exact with respect to some closed 1-form $\theta$, then it is LCC with Lee form $\theta$. 
The converse might not be true, as we shall see in Example \ref{ExEinstein}.

%%%%%%%%%%%%%%%%%%%%%%%%%%%%%%%%%%%%%%%%%%%%%%%%%%%%%%%%%%%%%%%%%%%%%%%%%%%%%%%%%%%%%%%%%
%%%%%%%%%%%%%%%%%%%%%%%%%%%%%%%%%%%%%%%%%%%%%%%%%%%%%%%%%%%%%%%%%%%%%%%%%%%%%%%%%%%%%%%%%
%																LCC FIRST AND SECOND KIND 
%%%%%%%%%%%%%%%%%%%%%%%%%%%%%%%%%%%%%%%%%%%%%%%%%%%%%%%%%%%%%%%%%%%%%%%%%%%%%%%%%%%%%%%%%
%%%%%%%%%%%%%%%%%%%%%%%%%%%%%%%%%%%%%%%%%%%%%%%%%%%%%%%%%%%%%%%%%%%%%%%%%%%%%%%%%%%%%%%%%

\section{LCC G$_2$-structures of the first and of the second kind}\label{LCCfirst}
A special class of exact LCC G$_2$-structures can be introduced after some considerations on the infinitesimal automorphisms. 

Recall that the {\em automorphism group} of a seven-dimensional manifold $M$ endowed with a G$_2$-structure $\f$ is 
\[
\mathrm{Aut}(M,\f) \coloneqq \left\{F\in\mathrm{Dif{}f}(M) \st F^*\f=\f \right\}. 
\]
Clearly, $\mathrm{Aut}(M,\f)$ is a closed Lie subgroup of the isometry group $\mathrm{Iso}(M,g_\f)$ of the Riemannian manifold $(M,g_\f)$. 
Moreover, its Lie algebra is given by 
\[
\mathfrak{aut}(M,\f) \coloneqq \left\{X\in\mathfrak{X}(M) \st \mathcal{L}_{\sst X}\f=0 \right\}, 
\]
and every {\em infinitesimal automorphism} $X\in\mathfrak{aut}(M,\f) $ is a Killing vector field for $g_\f$. 

If $\f$ is closed and $X\in\mathfrak{aut}(M,\f)$, then the 2-form $\iota_{\sst X}\f\in\Omega^2_7(M)$ is easily seen to be harmonic. 
When $M$ is compact, this implies that $\mathfrak{aut}(M,\f)$ is Abelian with dimension bounded by $\mathrm{min}\{6,b_2(M)\}$ (see \cite{PoRa}). 

Let us now focus on the case when $\f$ is LCC with Lee form $\theta$ not identically vanishing.  
For every infinitesimal automorphism $X\in\mathfrak{aut}(M,\f)$, we have
\[
0 = d(\mathcal{L}_{\sst X}\f) = \mathcal{L}_{\sst X}d\f =  \mathcal{L}_{\sst X}\theta\W\f, 
\]
whence we see that $\mathcal{L}_{\sst X}\theta = 0$. 
Consequently, $\theta(X)$ is constant and the map 
\[
\ell_\theta:\mathfrak{aut}(M,\f) \rightarrow \R,\quad \ell_\theta(X)\coloneqq\theta(X), 
\]
is a well-defined morphism of Lie algebras. 
This suggests that various meaningful ideas of locally conformal symplectic geometry (e.g.~\cite{ABP,Ban,BaMa,Vai}) make sense for LCC G$_2$-structures, too. 
In particular, as the map $\ell_\theta$ is either identically zero or surjective, we can give the following G$_2$-analogue of a definition first introduced by Vaisman in \cite{Vai}.  
\begin{definition}
A LCC $\G_2$-structure $(\f,\theta)$ is of the \emph{first kind} if the Lie algebra morphism $\ell_\theta$ is surjective, while it is of the \emph{second kind} otherwise. 
\end{definition}

If there exists at least one point $p$ of $M$ where $\theta|_p=0$, then the LCC G$_2$-structure $\f$ is necessarily of the second kind. 
As a consequence, if $\f$ is a LCC G$_2$-structure with Lee form $\theta$ such that $\theta|_p=df|_p$ for some smooth function $f\in\mathcal{C}^\infty(M)$, then 
the 3-form $e^{-f}\f$ defines a LCC G$_2$-structure of the second kind, as the corresponding Lee form is $\theta-df$. 
Hence, being of the first kind is not an invariant of the conformal class of $\f$.

Assume now that $\f$ is a LCC G$_2$-structure of the first kind.  
Then, its Lee form $\theta$ is nowhere vanishing and, consequently, $\chi(M)=0$ if $M$ is compact.  
Let us consider an infinitesimal automorphism $U\in\mathfrak{aut}(M,\f)$ such that $\theta(U)=-1$. 
The condition $\mathcal{L}_{\sst U}\f=0$ is equivalent to
\[
\f =  d\sigma  -\theta\W  \sigma, 
\]
where $\sigma\coloneqq\iota_{\sst U}\f\in\Omega^2_7(M)$. 
Thus, a LCC G$_2$-structure of the first kind is always exact. 
More precisely, it belongs to the image of the restriction of $d_\theta$ to $\Omega^2_7(M)$. 

We shall say that an exact G$_2$-structure $\f$ is {\em of the first kind} if it can be written as $\f = d_\theta(\iota_{\sst X}\f)$ with $\theta(X)=-1$. 
\begin{proposition}\label{exactfirstK}
Let $(\f,\theta)$ be a LCC $\G_2$-structure. Then, $\f=d_\theta(\iota_{\sst X}\f)$ if and only if $\mathcal{L}_{\sst X}\f = (\theta(X)+1)\f.$ 
In particular, $\f$ is of the first kind if and only if $\theta(X)=-1$. 
\end{proposition}
\begin{proof}
The first assertion follows from the identity 
\[
d_\theta(\iota_{\sst X}\f) = d(\iota_{\sst X}\f) - \theta\W\iota_{\sst X}\f = \mathcal{L}_{\sst X}\f -\theta(X)\f. 
\]
The second assertion is an immediate consequence of the above definition. 
\end{proof}

Some examples of LCC G$_2$-structures of the first and of the second kind will be discussed in Section \ref{exsect}. 
In particular, we will see that there exist exact G$_2$-structures of the form $\f=d_\theta\sigma$ with $\sigma\not\in\Omega^2_7(M)$.  

In \cite[Thm.~6.4]{FFR}, the structure of compact 7-manifolds admitting a LCC G$_2$-structure satisfying suitable properties was described. 
In view of the definitions introduced in this section, we can rewrite the statement of this structure theorem as follows.  
\begin{theorem}[\cite{FFR}]\label{StructResLCC}
Let $M$ be a compact seven-dimensional manifold endowed with a LCC $\G_2$-structure $(\f,\theta)$ of the first kind. If the $g_\f$-dual vector field $\theta^\sharp$ of $\theta$ belongs to $\mathfrak{aut}(M,\f)$, then
\begin{enumerate}[1)]
\item  $M$ is the total space of a fibre bundle over $\mathbb{S}^1$, and each fibre is endowed with a coupled $\SU(3)$-structure; 
\item  $M$ has a LCC $\G_2$-structure $\hat \varphi$ such that $d \hat \varphi = \hat \theta \wedge \hat\varphi$,  
where $\hat \theta$ is a 1-form with integral periods.
\end{enumerate}
\end{theorem}

Motivated by the structure results for locally conformal symplectic structures of the first kind obtained in \cite{Ban,BaMa}, we state the following more general problem. 
\begin{question}
What can one say about the structure of a (compact) 7-manifold $M$ endowed with a LCC G$_2$-structure of the first kind? 
\end{question}

We conclude this section mentioning a mild issue related to the above statement. 
In order to prove Theorem \ref{StructResLCC}, one deforms the given LCC $\G_2$-structure to one which gives $M$ the claimed structure of a bundle over $\mathbb{S}^1$ whose fibres are equipped with 
a coupled $\SU(3)$-structure, and the deformed structure has nothing to do with the given one. 
This kind of issue appears also in cosymplectic and in locally conformal symplectic geometry; results similar in spirit to Theorem \ref{StructResLCC} were obtained by Li \cite{Li} in the cosymplectic case,  
and by Banyaga \cite{Ban} in the locally conformal symplectic case. 
A different approach, which does not deform the given structure, was taken in \cite{GMP} for the cosymplectic case and in \cite{BaMa} for the locally conformal symplectic case: 
the same structure result holds, with the given structure, provided that the codimension-one foliation defined by the 1-form $\theta$ has one compact leaf.

%%%%%%%%%%%%%%%%%%%%%%%%%%%%%%%%%%%%%%%%%%%%%%%%%%%%%%%%%%%%%%%%%%%%%%%%%%%%%%%%%%%%%%%%%
%%%%%%%%%%%%%%%%%%%%%%%%%%%%%%%%%%%%%%%%%%%%%%%%%%%%%%%%%%%%%%%%%%%%%%%%%%%%%%%%%%%%%%%%%
%																LIE ALGEBRAS
%%%%%%%%%%%%%%%%%%%%%%%%%%%%%%%%%%%%%%%%%%%%%%%%%%%%%%%%%%%%%%%%%%%%%%%%%%%%%%%%%%%%%%%%%
%%%%%%%%%%%%%%%%%%%%%%%%%%%%%%%%%%%%%%%%%%%%%%%%%%%%%%%%%%%%%%%%%%%%%%%%%%%%%%%%%%%%%%%%%

\section{Lie algebras with a LCC G$_2$-structure}\label{LieAlgLCC}
We begin this section recalling a few basic facts on Lie algebras, in order to introduce some notations. 
Then, we focus on the construction of Lie algebras admitting a LCC G$_2$-structure,  
and we show a structure result for Lie algebras with an exact LCC G$_2$-structure. 
All Lie algebras considered in this section are assumed to be real. 

\subsection{Rank-one extension of Lie algebras}
Let $\frh$ be a Lie algebra of dimension $n$, denote by $[\cdot,\cdot]_\frh$ its Lie bracket, and by $d_\frh$ the corresponding Chevalley-Eilenberg differential. 
The structure equations of $\frh$ with respect to a basis $(e_1,\ldots,e_n)$ are given by
\[
[e_i,e_j]_\frh = \sum_{k=1}^nc^k_{ij}e_k, \qquad 1\leq i < j \leq n, 
\]
with $c^k_{ij}\in\R$, $c^k_{ij} = -c^k_{ji}$, and $\sum_{r=1}^n\left(c^r_{ij}c^s_{rk}+c^r_{jk}c^s_{ri}+c^r_{ki}c^s_{rj}\right)=0$.  
Equivalently, if $(e^1,\ldots,e^n)$ is the dual basis of  $(e_1,\ldots,e_n)$, then the structure equations of $\frh$ can be written as follows
\[
d_\frh e^k = -\sum_{1\leq i<j\leq n}c^k_{ij} e^i\W e^j, \qquad 1\leq k \leq n. 
\]
A Lie algebra $\frh$ is then described up to isomorphism by the $n$-tuple $(d_\frh e^1,\ldots,d_\frh e^n)$. 

The {\em rank-one extension} of $\frh$ induced by a derivation $D\in\mathrm{Der}(\frh)$ is the $(n+1)$-dimensional Lie algebra given by the vector space $\frh\oplus\R$ endowed with the Lie bracket 
\[
\left[(X,a),(Y,b)\right]\coloneqq \left([X,Y]_\frh + a\,D(Y)-b\,D(X),0\right), 
\]
for all $(X,a),~(Y,b)\in\frh\oplus\R$. We shall denote this Lie algebra by $\frh {\rtimes_D} \R$. 
Moreover, we let $\xi\coloneqq(0,1)$, and we denote by $\eta$ the 1-form on $\frh {\rtimes_D} \R$ such that $\eta(\xi)=1$ and $\eta(X)=0$, for all $X\in\frh$. 
Notice that if $\frh$ is a nilpotent Lie algebra and $D$ is a nilpotent derivation, then $\frh{\rtimes_D}\R$ is nilpotent.

Let $d$ denote the Chevalley-Eilenberg differential on $\frh {\rtimes_D} \R$. 
Using the Koszul formula, it is possible to check that for every $k$-form $\gamma\in\Lambda^k(\frh^*)$ the following identity holds:
\begin{equation}\label{CEd}
d\gamma = d_\frh \gamma +(-1)^{k+1}D^*\gamma\W\eta,
\end{equation}
where the natural action of an endomorphism $A\in\mathrm{End}(\frh)$ on $\Lambda^k(\frh^*)$ is given by 
\[
A^*\gamma(X_1, \ldots, X_k) = \gamma(AX_1,\ldots,X_k) +\cdots+ \gamma(X_1,\ldots,AX_k),
\]
for all $X_1,\ldots,X_k\in\frh$. Moreover, it is clear that $d\eta=0$. 

\subsection{A structure result for Lie algebras with an exact LCC G$_{\mathbf2}$-structure} 
Let $\frh$ be a six-dimensional Lie algebra. A pair $(\omega,\psi)\in\Lambda^2(\frh^*)\times\Lambda^3(\frh^*)$ defines an SU(3)-structure on $\frh$ if there exists a basis $(e^1,\ldots,e^6)$ of $\frh^*$ such that 
\begin{equation}\label{SU3adapted}
\omega  =e^{12}+e^{34}+e^{56},\quad \psi =   e^{135} - e^{146} - e^{236} - e^{245}. 
\end{equation}
We shall call $(e^1,\ldots,e^6)$ an {\em $\SU(3)$-basis} for $(\frh,\omega,\psi)$. 
An SU(3)-structure $(\omega,\psi)$ on $\frh$ is {\em half-flat} if $d_\frh\omega\W\omega=0$ and $d_\frh\psi=0$. 
A half-flat SU(3)-structure satisfying the condition $d_\frh\omega=c\psi$ for some $c\in\R$
is {\em coupled} if $c\neq0$, while it is {\em symplectic half-flat} if $c=0$.  

Similarly, a 3-form $\f$ on a seven-dimensional Lie algebra $\frg$ defines a G$_2$-structure if there is a basis $(e^1,\ldots,e^7)$ of $\frg^*$ such that 
\[
\f = e^{127} + e^{347}+ e^{567} + e^{135} - e^{146} - e^{236} - e^{245}. 
\]
We shall refer to $(e^1,\ldots,e^7)$ as a {\em $\G_2$-basis} for $(\frg,\f)$. 
A G$_2$-structure $\f$ is {\em closed} if $d_\frg\f=0$, while it is {\em locally conformal closed} ({\em LCC}\,) if $d_\frg\f=\theta\W\f$ for some 1-form $\theta\in\Lambda^1(\frg^*)$ with $d_\frg\theta=0$. 

If $\frh {\rtimes_D} \R$ is the rank-one extension of a six-dimensional Lie algebra $\frh$ endowed with an SU(3)-structure $(\omega,\psi)$, then it admits a G$_2$-structure defined by the 3-form 
\[
\f = \omega \W \eta +\psi.
\]
Indeed, if $(e^1,\ldots,e^6)$ is an SU(3)-basis for $(\frh,\omega,\psi)$, then  $(e^1,\ldots,e^6,e^7)$ with $e^7\coloneqq\eta$ is a G$_2$-basis for $(\frh {\rtimes_D} \R,\f)$. 

In the next proposition, we collect some conditions guaranteeing the existence of a LCC G$_2$-structure on the rank-one extension of a six-dimensional Lie algebra. 
For the sake of convenience, from now on we shall denote the Chevalley-Eilenberg differential on seven-dimensional Lie algebras simply by $d$. 

\begin{proposition}\label{PropCpdLCC}
Let $\frh$ be a six-dimensional Lie algebra endowed with a coupled $\SU(3)$-structure $(\omega,\psi)$ with $d_\frh\omega=c\psi$, and consider the rank-one extension $\frh {\rtimes_D} \R$, $D\in\mathrm{Der}(\frh)$,  
endowed with the $\G_2$-structure  $\f \coloneqq \omega\W\eta+\psi$.  Then: 
\begin{enumerate}[i)]
\item\label{prop1} $\f$ is LCC with Lee form $\theta = a\eta$, for some $a\in\R$, if and only if $D^*\psi = -(a+c)\psi$. 
In particular, it is closed if and only if $D^*\psi = -c\,\psi$;
\item\label{prop2} if $D^*\omega=\mu\omega$ with $\mu\neq-c$, then $\f$ is $d_{(-(c+\mu)\eta)}$-exact. 
Moreover, it is of the first kind if and only if $\mu=0$.
\end{enumerate}
\end{proposition}

\begin{proof}
Using \eqref{CEd}, we see that the G$_2$-structure $\f=\omega\W\eta+\psi$ is LCC with Lee form $\theta = a\eta$ if and only if
\[
a\eta\W\psi = a\eta\W\f = d(\omega\W\eta+\psi) = d_\frh\omega\W\eta +d_\frh\psi +D^*\psi\W\eta = \left(c\psi+D^*\psi\right)\W\eta. 
\]
From this, \ref{prop1}) follows.

As for \ref{prop2}), we first observe that the hypothesis $D^*\omega=\mu\omega$ implies 
\[
D^*\psip = \frac{1}{c}\, D^* d_\frh \omega = \frac{1}{c}\, d_\frh D^* \omega = \frac{\mu}{c}\, d_\frh\omega = \mu \psi. 
\]
Thus, $\f$ is LCC with Lee form $\theta = -(c+\mu)\eta$ by point \ref{prop1}). Moreover, 
\[
d\omega = d_\frh\omega - D^*\omega\W\eta  = c\psi - \mu\omega\W\eta. 
\]
Consequently, we get
\[
\f = \omega\W\eta+\psi = \omega\W\eta +\frac{1}{c}\left(d\omega + \mu\omega\W\eta\right) = d\left(\frac{\omega}{c} \right) + (c+\mu)\eta\W\frac{\omega}{c}. 
\]
Hence, $\f=d_{(-(c+\mu)\eta)}\left(\frac{\omega}{c}\right)$ is exact. Notice that $\frac{\omega}{c} = \iota_{\sst \frac{\xi}{c}}\f\in\Lambda^2_7((\frh {\rtimes_D} \R)^*)$. 
Therefore, according to Proposition \ref{exactfirstK}, $\f$ is of the first kind if and only if 
\[
0 = \theta\left(\frac{\xi}{c}\right) +1 = -(c+\mu)\,\eta\left(\frac{\xi}{c}\right) +1 =-\frac{\mu}{c}. 
\]
\end{proof}

\begin{remark} \ 
\begin{enumerate}[1)]
\item Proposition \ref{PropCpdLCC} generalizes some results obtained by the second author in the joint works \cite{FFR,FiRa1}. 
In detail, \cite[Prop.~5.1]{FFR} corresponds to point \ref{prop1}) with $a=-c$, while \cite[Prop.~4.2]{FiRa1} corresponds to point \ref{prop1}) with $a=c$.
\item When the SU(3)-structure $(\omega,\psi)$ on $\frh$ is symplectic half-flat and $D\in\mathrm{Der}(\frh)$ satisfies $D^*\psi=0$, then $\f = \omega\W\eta+\psi$ is a closed G$_2$-structure on $\frh {\rtimes_D} \R$ 
by point \ref{prop1}). This was already observed by Manero in  \cite[Prop.~1.1]{Man}.  
\item Recall that for a six-dimensional Lie algebra $\frh$ endowed with an SU(3)-structure $(\omega,\psi)$ the following isomorphisms hold: 
\[
\left\{A\in\mathrm{End}(\frh) \st A^*\omega=0 \right\} \cong \mathfrak{sp}(6,\R),\quad\left\{A\in\mathrm{End}(\frh) \st A^*\psi=0 \right\} \cong \mathfrak{sl}(3,\C)\subset \mathfrak{gl}(6,\R).
\]
In particular, if $(\omega,\psi)$ is coupled and $A^*\omega=0$, then $A\in\mathfrak{sp}(6,\R)\cap\mathfrak{sl}(3,\C) = \mathfrak{su}(3)$. \end{enumerate}
\end{remark}

The next result is the converse of point \ref{prop2}) of Proposition \ref{PropCpdLCC}. 
\begin{proposition}\label{PropLCCCpd}
Let $\frg$ be a seven-dimensional Lie algebra endowed with an exact LCC $\G_2$-structure $\f = d\sigma-\theta\W\sigma$, where $\theta\in\Lambda^1(\frg^*)$ is closed and $\sigma\in\Lambda^2_7(\frg^*)$. 
Assume that the non-zero vector $X\in\frg$ for which $\sigma = \iota_{\sst X}\f$ satisfies $\theta(X)\neq0$. 
Then, $\frg$ splits as a $g_\f$-orthogonal direct sum $\frg=\frh \oplus \R$, where  $\R=\langle X\rangle$ and $\frh\coloneqq \ker(\theta)$ is a six-dimensional ideal endowed with a coupled $\SU(3)$-structure 
$(\omega,\psi)$ induced by $\f$. Moreover, there is a derivation $D\in\mathrm{Der}(\frh)$ such that $D^*\omega=-\left(1+\theta(X)\right)\omega$, and $\frg\cong \frh {\rtimes_D}\R$.
\end{proposition}
\begin{proof}
It is clear that $\frh\coloneqq\mathrm{ker}(\theta)$ is a six-dimensional ideal of $\frg$, as $\theta\in\Lambda^1(\frg^*)$ is non-zero and closed. 
Since $\theta(X)\neq0$, we see that the vector space $\frg$ decomposes into the direct sum $\frg = \frh \oplus \R$, with  $\R=\langle X\rangle$. 
The $\R$-linear map 
\[
D:\frh\rightarrow\frh,\quad H \mapsto [X,H],
\]
is well-defined, as $d\theta=0$, and it is a derivation of $\frh$ by the Jacobi identity. From this it is easy to see that $\frg\cong \frh{\rtimes_D}\R$ as a Lie algebra. 

Let $\theta^\sharp\in\frg$ be the $g_\f$-dual vector of $\theta$. By definition, $\theta(\theta^\sharp) = g_\f(\theta^\sharp,\theta^\sharp) = |\theta|^2\neq0$. 
Thus, $\theta^\sharp\in\langle X\rangle\subset \frg$ and the decomposition $\frg = \frh \oplus \R$ is $g_\f$-orthogonal, i.e.,  $g_\f(H,X)=0$ for all $H\in\frh$. 
Consequently, depending on the choice of a unit vector $\varepsilon\frac{X}{|X|}\in\langle X\rangle$, with $\varepsilon\in\{\pm1\}$, the ideal $\frh$ admits an SU(3)-structure defined by the pair
\[
\omega\coloneqq\left.\left(\iota_{\varepsilon\sst\frac{X}{|X|}}\f\right)\right|_\frh, \quad \psi\coloneqq \left.\f\right|_\frh, 
\]  
Notice that $\omega = \left.\varepsilon|X|^{-1}\sigma\right|_\frh = \varepsilon|X|^{-1}\sigma$, as $\iota_{\sst X}\sigma=0$. 
We claim that $(\omega,\psi)$ is coupled with coupling constant $c=\varepsilon|X|^{-1}.$ First, observe that for all $H_1,H_2,H_3\in\frh$ we have
\[
\psi(H_1,H_2,H_3) = (d\sigma-\theta\W\sigma)(H_1,H_2,H_3) = d\sigma(H_1,H_2,H_3) = d_\frh\sigma(H_1,H_2,H_3). 
\]
Therefore, $d_\frh\omega = \varepsilon|X|^{-1}\psi$ and the claim is proved.  
Let us now determine the expression of $(D^*\sigma)|_\frh$, from which we will deduce the expression of $D^*\omega$. For all $H_1,H_2\in\frh$, we have 
\[
D^*\sigma(H_1,H_2)	 = \sigma([X,H_1],H_2) - \sigma([X,H_2],H_1) = -d\sigma(X,H_1,H_2) = -(\iota_{\sst X}d\sigma)(H_1,H_2), 
\]
where the second equality follows from Koszul formula and the condition $\iota_{\sst X}\sigma=0$. 
Since $\f = d\sigma-\theta\W\sigma$, on $\frh$ we have
\[
D^*\sigma = -\iota_{\sst X}d\sigma = -\iota_{\sst X}(\f+\theta\W\sigma)  = -\left(1+\theta(X)\right)\sigma.
\]
Thus, 
\[
D^*\omega = \varepsilon|X|^{-1}D^*\sigma = -\left(1+\theta(X)\right)\omega. 
\]
\end{proof}

Combining propositions \ref{PropCpdLCC} and \ref{PropLCCCpd}, we obtain the following analogue of \cite[Thm.~1.4]{ABP} for exact locally conformal symplectic Lie algebras. 
\begin{theorem}\label{ThmLCCex}
There is a one-to-one correspondence between seven-dimensional Lie algebras $\frg$ admitting an exact $\G_2$-structure of the form $\f=d\sigma-\theta\W\sigma$, with $\sigma=\iota_{\sst X}\f\in\Lambda^2_7(\frg^*)$ 
and $\theta(X)\neq0$, and six-dimensional Lie algebras $\frh$ endowed with a coupled $\SU(3)$-structure $(\omega,\psi)$, with coupling constant $c$, and a derivation $D\in\mathrm{Der}(\frh)$ 
such that $D^*\omega=\mu\omega$, for some $\mu\neq-c$. 
\end{theorem}

According to a result of Dixmier (see \cite[Th\'eor\`eme 1]{Dixmier}), the Lichnerowicz cohomology of a nilpotent Lie algebra with respect to any closed 1-form vanishes. 
Hence, every LCC $\G_2$-structure on a seven-dimensional nilpotent Lie algebra is exact. 
We can use this observation to show that among the seven-dimensional non-Abelian nilpotent Lie algebras admitting closed G$_2$-structures (see \cite{CoFe} for the classification) there is no one 
where LCC $\G_2$-structures exist. 
\begin{proposition}\label{CnoLCC}
None of the seven-dimensional non-Abelian nilpotent Lie algebras admitting closed $\G_2$-structures admits LCC $\G_2$-structures. 
\end{proposition}
\begin{proof}
By \cite{CoFe}, a seven-dimensional non-Abelian nilpotent Lie algebra admitting closed G$_2$-structures is isomorphic to one of the following: 
\[
\renewcommand\arraystretch{1.2}
\begin{array}{rcl}
\frn_1 	&=& (0, 0, 0, 0, e^{12}, e^{13}, 0),\\
\frn_2 	&=& (0, 0, 0, e^{12}, e^{13}, e^{23}, 0),\\
\frn_3 	&=& (0, 0, e^{12}, 0, 0, e^{13} + e^{24}, e^{15}),\\
\frn_4	&=& (0, 0, e^{12}, 0, 0, e^{13}, e^{14} + e^{25}),\\
\frn_5 	&=& (0, 0, 0, e^{12}, e^{13}, e^{14}, e^{15}),\\
\frn_6 	&=& (0, 0, 0, e^{12}, e^{13}, e^{14} + e^{23}, e^{15}),\\
\frn_7 	&=& (0, 0, e^{12}, e^{13}, e^{23}, e^{15} + e^{24}, e^{16} + e^{34}),\\
\frn_8 	&=& (0, 0, e^{12}, e^{13}, e^{23}, e^{15} + e^{24}, e^{16} + e^{34} + e^{25}),\\
\frn_9	&=& (0, 0, e^{12}, 0, e^{13} + e^{24}, e^{14}, e^{46} + e^{34} + e^{15} + e^{23}),\\
\frn_{10} 	&=& (0, 0, e^{12}, 0, e^{13}, e^{24} + e^{23}, e^{25} + e^{34} + e^{15} + e^{16} - 3 e^{26}),\\
\frn_{11} 	&=& (0, 0, 0, e^{12}, e^{23},-e^{13}, 2 e^{26} - 2 e^{34} - 2 e^{16} + 2 e^{25}).
\end{array}
\renewcommand\arraystretch{1}
\]
To show the proposition, we will use Dixmier's result together with the following fact: 
a 3-form $\phi$ on a seven-dimensional Lie algebra $\frg$ defines a G$_2$-structure if and only if the symmetric bilinear map 
\[
b_\phi: \frg\times\frg \rightarrow \Lambda^7(\frg^*)\cong \frg,\quad (X,Y)\mapsto \frac16\,\iota_{\sst X} {\phi} \W \iota_{\sst Y} {\phi} \W {\phi}, 
\]
is definite (cf.~\cite{Hit}). 
Now, for every nilpotent Lie algebra $\frn_i$ appearing above, we consider the generic closed 1-form $\theta = \sum_{k=1}^7\theta_ke^k\in\Lambda^1(\frn_i^*)$, 
with some of the real numbers $\theta_k$ possibly zero as $d\theta=0$, and the generic $d_\theta$-exact 3-form $\phi = d\sigma-\theta\W\sigma$, 
where $\sigma = \sum_{1\leq j < k\leq 7} \sigma_{jk}e^{jk}\in\Lambda^2(\frn_i^*)$. 
Then, we compute the map $b_\phi$ associated with such a 3-form $\phi$, and we observe that in each case 
it cannot be definite. Indeed, it is just a matter of computation to show that $b_\phi(e_6,e_6)=0$ for the nilpotent Lie algebras $\frn_i$, with $i=1,2,3,4,5,6$, 
and that $b_\phi(e_7,e_7)=0$ for the remaining ones. 
\end{proof}

%%%%%%%%%%%%%%%%%%%%%%%%%%%%%%%%%%%%%%%%%%%%%%%%%%%%%%%%%%%%%%%%%%%%%%%%%%%%%%%%%%%%%%%%%
%%%%%%%%%%%%%%%%%%%%%%%%%%%%%%%%%%%%%%%%%%%%%%%%%%%%%%%%%%%%%%%%%%%%%%%%%%%%%%%%%%%%%%%%%
%																EXAMPLES 
%%%%%%%%%%%%%%%%%%%%%%%%%%%%%%%%%%%%%%%%%%%%%%%%%%%%%%%%%%%%%%%%%%%%%%%%%%%%%%%%%%%%%%%%%
%%%%%%%%%%%%%%%%%%%%%%%%%%%%%%%%%%%%%%%%%%%%%%%%%%%%%%%%%%%%%%%%%%%%%%%%%%%%%%%%%%%%%%%%%

\section{Examples}\label{exsect}
We now use the results of the previous section to construct various examples of LCC G$_2$-structures that clarify the interplay between the conditions discussed in sections \ref{LCCsect} and \ref{LCCfirst}.  

First of all, we need to start with a six-dimensional Lie algebra admitting coupled SU(3)-structures. In the nilpotent case, the following classification is known (see \cite[Thm.~4.1]{FiRa1}). 
\begin{theorem}[\cite{FiRa1}]
Up to isomorphism, a six-dimensional non-Abelian nilpotent Lie algebra admitting coupled $\SU(3)$-structures is isomorphic to one of the following
\[
\frh_{1} = \left(0,0,0,0,e^{14}+e^{23},e^{13}-e^{24}\right),\quad \frh_{2} =  \left(0,0,0,e^{13},e^{14}+e^{23},e^{13}-e^{15}-e^{24}\right).
\]
In both cases, $(e^1,\ldots,e^6)$ is an $\SU(3)$-basis for a certain coupled structure $(\omega,\psi)$. 
\end{theorem}

Let us consider the coupled SU(3)-structure $(\omega,\psi)$ on $\frh_1$. 
Since $(e^1,\ldots,e^6)$ is an SU(3)-basis, the forms $\omega$ and $\psi$ can be written as in \eqref{SU3adapted}, and a simple computation shows that $d_{\frh_1}\omega=-\psi$. 
As observed in \cite{FiRa1}, the inner product $g=\sum_{i=1}^6(e^i)^2$ induced by $(\omega,\psi)$ is a {\em nilsoliton}, i.e., its Ricci operator is of the form 
\begin{equation}\label{RS}
\mathrm{Ric}(g) = -3\,\mathrm{Id} +4D_1, 
\end{equation}
where $D_{1}\in\mathrm{Der}(\frh_1)$ is given by
\[
D_1(e_1) = \frac12 e_1,~D_1(e_2) = \frac12 e_2,~D_1(e_3) = \frac12 e_3,~D_1(e_4) = \frac12 e_4,~D_1(e_5) = e_5,~D_1(e_6) = e_6,
\]
$(e_1,\ldots,e_6)$ being the basis of $\frh_1$ whose dual basis is the SU(3)-basis of $(\frh_1,\omega,\psi)$. 
For more details on nilsolitons we refer the reader to \cite{Lau}. 

We know that the rank-one extension $\frh_1{\rtimes_D}\R$ of $\frh_1$ induced by a derivation $D\in\mathrm{Der}(\frh_1)$ admits a G$_2$-structure defined by the 3-form 
$\f=\omega\W\eta+\psi$, and that the G$_2$-basis is given by $\left(e^1,\ldots,e^6,e^7\right)$ with $e^7\coloneqq\eta$. 
In what follows, we shall always write the structure equations of $\frh_1{\rtimes_D}\R$ with respect to such a basis. 

The first example we consider was discussed in \cite{FiRa1}. It consists of a solvable Lie algebra endowed with a LCC G$_2$-structure $\f$ inducing an Einstein inner product. 
As we will see, $\f$ is not exact, that is, its class $[\f]_\theta$ in the Lichnerowicz cohomology is not zero.
\begin{example}\label{ExEinstein}
Let us consider the derivation $D_{1}\in\mathrm{Der}(\frh_1)$ appearing in \eqref{RS}. 
The rank-one extension $\frh_1 {\rtimes_{D_1}}\R$ of $\frh_1$ has structure equations
\[
\left(\frac12 e^{17},\frac12 e^{27},\frac12 e^{37},\frac12 e^{47},e^{14}+e^{23}+e^{57},e^{13}-e^{24}+e^{67},0 \right). 
\]
Since $D_1^*\psi= 2\psi$ and the coupling constant is $c=-1$, 
the G$_2$-structure $\f = \omega\W\eta+\psi$ on $\frh_1 {\rtimes_{D_1}}\R$ is LCC with Lee form $\theta = -\eta$, by point \ref{prop1}) of Proposition \ref{PropCpdLCC}. 
Moreover, it induces the inner product $g_\f = g + \eta^2$, which is Einstein with Ricci operator $\mathrm{Ric}(g_\f) =  -3\,\mathrm{Id}$ by \cite[Lemma 2]{Lau1}. 
A simple computation shows that $\f$ cannot be equal to $d_\theta\sigma$ for any 2-form $\sigma\in  \Lambda^2((\frh_1 {\rtimes_{D_1}}\R)^*)$. 
In particular, it is of the second kind. 

We conclude this example observing that the Lie algebra $\frh_1 {\rtimes_{D_1}}\R$ is solvable and not unimodular, as $\mathrm{tr}(\mathrm{ad}_{e_7}) = \mathrm{tr}(D_1)=4.$ 
Thus, the corresponding simply connected solvable Lie group does not admit any compact quotient. 
\end{example}

The next two examples were obtained in \cite[Sect.~5]{FFR}. 
In the first one the LCC G$_2$-structure is of the first kind, while in the second one the LCC G$_2$-structure is exact but it is not of the first kind.  
\begin{example}[\cite{FFR}]\label{ex2}
Consider the derivation $D_2\in\mathrm{Der}(\frh_1)$ defined as follows 
\[
D_2(e_1) = -e_3,~D_2(e_2) = -e_4,~D_2(e_3) = e_1,~D_2(e_4) = e_2,~D_2(e_5) =0,~D_2(e_6)=0.
\]
Then, the rank-one extension $\frh_1{\rtimes_{D_2}}\R$ has structure equations 
\[
\left(e^{37},e^{47},-e^{17},-e^{27},e^{14}+e^{23},e^{13}-e^{24},0\right),
\]
and $D_2^*\omega=0$. Thus, by point \ref{prop2}) of Proposition \ref{PropCpdLCC}, we have that the 3-form $\f=\omega\W\eta+\psi$ defines a LCC G$_2$-structure of the first kind on $\frh_1{\rtimes_{D_2}}\R$ 
with Lee form $\theta=\eta$. 
\end{example}

\begin{example}[\cite{FFR}]\label{ex3}
Consider the rank-one extension $\frh_1{\rtimes_{D_3}}\R$, where $D_3\in\mathrm{Der}(\frh_1)$ is given by
\[
D_3(e_1) =  2 e_3,~D_3(e_2) = 2e_4,~D_3(e_3) = e_1,~D_3(e_4) = e_2,~D_3(e_5) =0,~D_3(e_6)=0. 
\]
The structure equations of $\frh_1{\rtimes_{D_3}}\R$ are the following
\[
\left(e^{37},e^{47},2e^{17},2e^{27},e^{14}+e^{23},e^{13}-e^{24},0\right).
\]
Since $D_3^*\psi=0$ but $D_3^*\omega\neq0$, 
the G$_2$-structure $\f = \omega\W\eta+\psi$ on $\frh_1 {\rtimes_{D_3}}\R$ is LCC with Lee form $\theta = \eta$, by point \ref{prop1}) of Proposition \ref{PropCpdLCC}. 
We observe that 
\[
\f = d_\theta\gamma,
\]
where $\gamma = \frac57 e^{12} - \frac37 e^{14} + \frac37 e^{23} - \frac17 e^{34} - e^{56}$ does not belong to $\Lambda^2_7((\frh_1{\rtimes_{D_3}}\R)^*)$. 
In this case, the only infinitesimal automorphisms of $\f$ are of the form $X = a\,e_5+b\,e_6 \in \frh_1{\rtimes_{D_3}}\R$, with $a,b\in\R$. Thus, $\f$ is of the second kind. 
\end{example}

\begin{remark}
As shown in \cite{FFR}, both the Lie algebras considered in examples \ref{ex2} and \ref{ex3} are solvable and unimodular, and 
the corresponding simply connected solvable Lie groups admit a lattice. 
Thus, both examples give rise to a compact seven-dimensional solvmanifold endowed with a LCC G$_2$-structure. 
\end{remark}

{\begin{remark}\label{exact-not-1stkind}
It was proved in \cite[Prop.~5.5]{BaMa} that on a unimodular Lie algebra every exact locally conformal symplectic structure is of the first kind. 
This is not the case in the G$_2$ setting: indeed, the LCC G$_2$-structure of Example \ref{ex3} is exact but not on the first kind, while the Lie algebra $\frh_1{\rtimes_{D_3}}\R$ is unimodular. 
\end{remark}

\medskip\noindent
{{\bf Acknowledgements.} The authors would like to thank Daniele Angella for useful conversations. 
The first author was supported by a \emph{Juan de la Cierva - Incorporaci\'on} Fellowship of Spanish Ministerio de Ciencia e Innovaci\'on. 
Both authors were partially supported by GNSAGA of INdAM - Istituto Nazionale di Alta Matematica.


\begin{thebibliography}{ASM}

\bibitem{ABP}
D.~Angella, G.~Bazzoni, and M.~Parton. 
\newblock Structure of locally conformally symplectic Lie algebras and solvmanifolds. 
\newblock To appear in  {\em Ann.~Scuola Norm.~Sup.~Pisa Cl.~Sci.} doi: 10.2422/2036-2145.201708\textunderscore015. 

\bibitem{Ban}
A.~Banyaga.
\newblock On the geometry of locally conformal symplectic manifolds.
\newblock In {\em Infinite dimensional {L}ie groups in geometry and representation theory ({W}ashington, {DC}, 2000)}, pages 79--91. World Sci.~Publ., River Edge, NJ, 2002.

\bibitem{Bazzoni}
G.~Bazzoni.
\newblock Locally conformally symplectic and K\"ahler geometry.
\newblock \href{https://arxiv.org/abs/1711.02440}{arXiv:1711.02440}.

\bibitem{BaMa}
G.~Bazzoni and J.~C.~Marrero.
\newblock On locally conformal symplectic manifolds of the first kind.
\newblock {\em Bull. Sci. Math.}, {\bf143}, 1--57, 2018.

\bibitem{Berger}
M.~Berger.
\newblock Sur les groupes d'holonomie homog\`ene des vari\'et\'es \`a connexion affine et des vari\'et\'es riemanniennes.
\newblock {\em Bull.~Sci.~Math.~France}, {\bf83}, 279--330, 1955.

\bibitem{Bry}
R.~L. Bryant.
\newblock Metrics with exceptional holonomy.
\newblock {\em Ann. of Math.}, {\bf126} (3), 525--576, 1987.

\bibitem{Bry1}
R.~L.~Bryant.
\newblock Some remarks on {G$_2$}-structures.
\newblock In {\em Proceedings of {G}{\"o}kova {G}eometry-{T}opology {C}onference 2005}, pages 75--109. G{\"o}kova Geometry/Topology Conference (GGT), G{\"o}kova, 2006.

\bibitem{ChSa}
S.~Chiossi and S.~Salamon.
\newblock The intrinsic torsion of {$\rm SU(3)$} and {G$_2$} structures.
\newblock In {\em Differential geometry, {V}alencia, 2001},  pp.  115--133. World Sci. Publ., River Edge, NJ, 2002. 

\bibitem{Chiossi-Fino}
S.~Chiossi and A.~Fino.
\newblock Conformally parallel {G}$_2$ structures on a class of solvmanifolds.
\newblock {\em Math. Z.}, {\bf 252} (4), 825--848, 2006.

\bibitem{CoFe}
D.~Conti and M.~Fern{{\'a}}ndez.
\newblock Nilmanifolds with a calibrated {G}$_2$-structure.
\newblock {\em Differ. Geom. Appl.}, {\bf 29} (4), 493--506, 2011.

\bibitem{Dixmier}
J.~Dixmier.
\newblock Cohomologie des alg\`ebres de {L}ie nilpotentes.
\newblock {\em Acta Sci. Math. Szeged}, {\bf 16}, 246--250, 1955.

\bibitem{Dragomir-Ornea}
S.~Dragomir and L.~Ornea.
\newblock Locally conformal K\"ahler geometry.
\newblock {\em Progress in Mathematics}, {\bf 155}, Birkh\"auser, Boston, 1998.

\bibitem{Eliashberg-Murphy}
Y.~Eliashberg and E.~Murphy.
\newblock Making cobordisms symplectic.
\newblock \href{https://arxiv.org/abs/1504.06312}{arXiv:1504.06312}.

\bibitem{FFR} 
M.~Fern\'andez, A.~Fino, and A.~Raf{}fero.
\newblock Locally conformal calibrated {G}$_2$-manifolds
\newblock  {\em Ann.~Mat.~Pura Appl.}, {\bf 195} (5), 1721--1736, 2016.

\bibitem{FeGr}
M.~Fern{{\'a}}ndez and A.~Gray.
\newblock Riemannian manifolds with structure group {G}$_{2}$.
\newblock {\em Ann. Mat. Pura Appl.}, {\bf 132}, 19--45, 1982.

\bibitem{FeUg}
M.~Fern{{\'a}}ndez and L.~Ugarte.
\newblock A differential complex for locally conformal calibrated {G}$_2$-manifolds.
\newblock {\em Illinois J. Math.}, {\bf44} (2), 363--390, 2000.

\bibitem{FiRa1}
A.~Fino and A.~Raf{}fero.
\newblock Einstein locally conformal calibrated {${\rm G}_2$}-structures.
\newblock {\em Math. Z.}, {\bf 280} (3-4), 1093--1106, 2015.

\bibitem{FiRa2}
A.~Fino and A.~Raf{}fero. 
\newblock Coupled SU$(3)$-structures and supersymmetry. 
\newblock {\em Symmetry} {\bf7} (2), 625--650, 2015.

\bibitem{Gray-Hervella} 
A.~Gray and L. Hervella, 
\newblock The sixteen classes of almost Hermitian manifolds and their linear invariants.
\newblock {\em Ann.~Mat.~Pura Appl.}, {\bf123} (4), 35--58, 1980.

\bibitem{GMP} 
V.~Guillemin, E.~Miranda and A.R.~Pires, 
\newblock Codimension one symplectic foliations and regular Poisson structures.
\newblock {\em Bull.~Braz.~Math.~Soc.~(N.~S.)}, {\bf42} (4), 607--623, 2011. 

\bibitem{Hit}
N.~Hitchin.
\newblock Stable forms and special metrics.
\newblock In {\em Global differential geometry: the mathematical legacy of {A}lfred {G}ray ({B}ilbao, 2000)}, v.~288 of {\em Contemp. Math.}, pp.~70--89. 
Amer. Math. Soc., 2001.

\bibitem{IPP}
S.~Ivanov, M.~Parton, and P.~Piccinni.
\newblock Locally conformal parallel {$\rm G_2$} and {${\rm Spin}(7)$} manifolds.
\newblock {\em Math. Res. Lett.}, {\bf 13} (2-3), 167--177, 2006. 

\bibitem{Lau}
J.~Lauret.
\newblock Ricci soliton homogeneous nilmanifolds.
\newblock {\em Math. Ann.}, {\bf 319} (4), 715--733, 2001.

\bibitem{Lau1}
J.~Lauret.
\newblock Standard {E}instein solvmanifolds as critical points.
\newblock {\em Q. J. Math.}, {\bf52} (4), 463--470, 2001.

\bibitem{Li}
H.~Li.
\newblock Topology of co-symplectic/co-K\"{a}hler manifolds, 
\newblock{\em Asian J. Math.} {\bf 12}, 527--544, 2008.

\bibitem{Man}
V.~Manero.
\newblock Construction of Lie algebras with special G$_2$-structures. 
\newblock \href{https://arxiv.org/abs/1507.07352v2}{arXiv:1507.07352}. 

\bibitem{MaCa}
F.~Mart{\'{\i}}n~Cabrera.
\newblock {${\rm SU}(3)$}-structures on hypersurfaces of manifolds with {$\rm G_2$}-structure.
\newblock {\em Monatsh. Math.}, {\bf148} (1), 29--50, 2006.

\bibitem{Ornea-Verbitsky}
L.~Ornea and M.~Verbitsky.
\newblock A report on locally conformally K\"ahler manifolds.
\newblock {\em Harmonic maps and differential geometry}, 135--149, Contemp. Math.~{\bf542}, Amer. Math. Soc., 2011.

\bibitem{PoRa}
F.~Podest\`a and A.~Raf{}fero. 
\newblock On the automorphism group of a closed G$_2$-structure. 
\newblock To appear in {\em Q.~J.~Math.} doi: 10.1093/qmath/hay045

\bibitem{Raf}
A.~Raf{}fero. 
\newblock Half-flat structures inducing {E}instein metrics on homogeneous spaces. 
\newblock {\em Ann. Global Anal. Geom.}, {\bf48} (1), 57--73, 2015.

\bibitem{Vai}
I.~Vaisman.
\newblock Locally conformal symplectic manifolds.
\newblock {\em Internat. J. Math. Math. Sci.}, {\bf8} (3), 521--536, 1985.

\bibitem{Ver}
M.~Verbitsky.
\newblock An intrinsic volume functional on almost complex 6-manifolds and nearly {K}{\"a}hler geometry.
\newblock {\em Pacific J. Math.}, {\bf 235} (2), 323--344, 2008.

\end{thebibliography}
\end{document}